\documentclass{amsart}
\usepackage{graphicx}
\usepackage{amssymb,amscd,amsthm,amsxtra}
\usepackage{latexsym}
\usepackage{epsfig}

\newtheorem{thm}{Theorem}[section]

\newtheorem{lem}[thm]{Lemma}
\newtheorem{prop}[thm]{Proposition}
\theoremstyle{definition}
\newtheorem{defn}[thm]{Definition}
\theoremstyle{remark}

\numberwithin{equation}{section}

\newcommand{\be}{\begin{equation}}
\newcommand{\ee}{\end{equation}}

\newcommand{\R}{\mathbb R}

\newcommand{\eps}{\varepsilon}
\newcommand{\Om}{\Omega}

\newcommand{\p}{\partial}

\newcommand{\comment}[1]{}

\begin{document}

\title[Viscosity solutions]{Viscosity solutions and the minimal surface system} %
\author{O. Savin}
\address{Department of Mathematics, Columbia University, New York, NY 10027}
\email{\tt  savin@math.columbia.edu}

\begin{abstract}
We give a definition of viscosity solution for the minimal surface system and prove a version of Allard regularity theorem in this setting.
\end{abstract}

\maketitle

\section{Introduction}

There are two main approaches to the theory of nonlinear elliptic scalar equations. One of them is variational (the $L^2$ approach), and it is based on energy estimates. This applies to equations with divergence structure. The second one, which regards more general nonlinear equations, is the viscosity solution approach (or $L^\infty$ approach), and it is based solely on the maximum principle. 

For the general theory of nonlinear elliptic systems only the variational approach seems to be successful. The reason is that the maximum principle does not extend to graphs when the codimension is higher than one.

In this short note we show that purely nonvariational techniques can be employed in the special situation of the minimal surface system.

Minimal submanifolds are usually studied from the geometric measure theory point of view. There are not many available results concerning the minimal surface system. This is due in part to the examples of Lawson and Osserman \cite{LO} which show a quite different situation with respect to the minimal surface equation. Uniqueness does not hold, and the existence of classical (Lipschitz) solutions to the Dirichlet problem with smooth data may fail as well. They also gave an example of nontrivial global Lipschitz solution to the Bernstein problem $u: \R^4 \to \R^3$ obtained as a suitable scaling of the Hopf map $\eta:S^3 \to S^2$, 
$$u(x)=\frac{\sqrt 5}{2}|x|\, \, \eta \left( \frac {x}{|x|} \right), \quad \quad \eta(z_1,z_2)=(|z_1|^2-|z_2|^2, 2 z_1 \bar{z_2}).$$
However, the existence of classical solutions and the Bernstein theorem  hold under specific bound assumptions involving the principal values of $Du$, see \cite{F,JX,W}. 

The $n$-dimensional area functional of the graph of a $C^1$ map $u:\Om \subset \R^n \to \R^m$,
$$\Gamma:=\{ (x,u(x))| x\in \Om\}  \subset \R^{n+m} $$
is given by
$$\mathcal H^n(\Gamma)=\int_\Om \left (det(I +  Du ^T Du)\right)^{1/2} dx.$$
If $\Gamma$ is critical for the $n$-dimensional area functional, then $u$ is solution to the minimal surface system
\be \label{D}
div(DF(Du))=0,
\ee
where $$F(A) = (\det (I + A^T A))^\frac 12, \quad \quad A \in \R^{n \times m}.$$
If $u \in C^2$, after expanding the divergence in \eqref{D} we can rewrite it as the following system of $m$ equations (using the summation index convention) 
\be\label{MS}
F_{\alpha i, \beta j} (Du) \, \, u^\beta_{ij} =0,
\ee
$$ 1 \le i,j \le n, \quad 1 \le \alpha, \beta \le m.$$
 The minimal surface system is invariant under rigid motions of its graph in $\R^{n +m}$. 
 Let $X_0=(x_0,u(x_0))$ be a point on $\Gamma$. After a rigid motion, let us assume for simplicity that $X_0=0$, $Du(0)=0$. Then \eqref{MS} simply becomes
$$\triangle u^\alpha (0) =0.$$

Let $\mathcal S \subset \R^{n+m}$ be a $C^2$ hypersurface that touches the graph $\Gamma$ at the origin so that $\Gamma$ stays on one-side of $\mathcal S$. Assume for simplicity of notation that the normal to $\mathcal S$ at $0$ points in the direction of the $u^1$ coordinate axis.  

We can view $\mathcal S$ as the zero level surface of a function $H(X)$ defined in $\R^{n+m}$, $$X=(x,z), \quad \quad x \in \R^n, \quad z \in \R^m,$$ with $\nabla H(0)$ parallel to the $z_1$ axis, and assume $\Gamma \subset \{H \le 0\}$. Differentiating twice
$$H(x,u(x)) \le 0, \quad H(0)=0,$$
and using $Du(0)=0$, $H_i(0)=0$, $H_{z_\alpha}(0)=0$ if $\alpha \ne 1$  we find 
$$ H_{ii} + H_{z_1}u^1_{ii} \le 0  \quad \Longrightarrow \quad \triangle _x H (0) \le 0.$$
This means that the Laplace of $H$ along the tangent space to $\Gamma$ at $X_0$ is nonpositive.

\begin{defn}\label{d1}
We say that $H(X)$ is a {\it comparison function} for the minimal surface system in the open set $U \subset \R^{n+m}$ if at any point $X \in U$ we have $$\triangle_L H (X) >0$$ for any $n$ dimensional vectorspace $L$ which is normal to $\nabla H(X)$. Here $\triangle_LH(X)$ denotes the Laplace operator of $H$ restricted the vectorspace $L$ passing through $X$.  
\end{defn}

The discussion above says that if $H$ is a comparison function in a neighborhood of $0$ then any $n$ dimensional $C^2$ minimal submanifold $\Gamma$ cannot be locally tangent to the hypersurface $\mathcal S=\{H=H(0)\}$ at $0$ and with $\Gamma \subset \{H \le H(0)\}$.
 
\begin{defn}
Let $u:\Om \to \R^m$ be a continuous function. We say that $u$ is a{\it viscosity solution} of the minimal surface system \eqref{MS} if its graph $\Gamma$ cannot touch the level set of a comparison function $\{H=c\}$ from the side $\{ H \le c\}$. 
\end{defn}

We have the following easy consequences directly from the definition of viscosity solutions.

\begin{prop}
$u \in C^2(\Om)$ is a viscosity solution if and only if $u$ is a classical solution of the system \eqref{MS}.
\end{prop}

\begin{prop}[Stability]
If $u_m$ is a sequence of viscosity solutions of the system \eqref{MS} and $u_m \to u$ uniformly on compact sets then $u$ is a viscosity solution as well.
\end{prop}
 The notion of viscosity solutions can be extended to $n$ dimensional compact sets $\Gamma$ in $\R^{n+m}$ instead of just graphs. In this setting, the viscosity solution definition is weaker than the one provided by the theory of varifolds. For example if $T$ is a stationary varifold then its support cannot be tangent to $\{H=c\}$ from the side $\{ H \le c\}$, with $H$ a comparison function. This follows from the fact that the projection onto the level surface $\{ H=c -\eps\}$ from the side $\{ H \ge c -\eps\}$ decreases the $n$ dimensional volume.  

Our main result is an $\eps$ regularity theorem for viscosity solutions. We show that if $u$ is sufficiently close (in $L^\infty$) to a linear map $l$ then it must be smooth in the interior.

\begin{thm}\label{T1}
Assume $u$ is a viscosity solution such that
$$|u-l| \le \eps \quad \mbox{in} \quad B_1,$$
where $l(x)=b+ Ax$ is a linear function from $\R^n$ to $\R^m$.

If $\eps \le \eps_0$ small, then $u \in C^2(B_{1/2})$ and
$$  \|D^2u\|_{C^{\alpha}(B_{1/2})} \le C \eps.$$
Here $\eps_0$ and $C$ depend only on $n$, $m$ and $|A|$.
\end{thm}

Analiticity of solutions follows then from the linear theory of elliptic systems. 

Theorem \ref{T1} can be viewed as a version of Allard regularity theorem (see \cite{A}). The method of proof is however different and it is based on nonvariational methods and the weak Harnack inequality. 

The estimate in Theorem \ref{T1} does not seem to follow from Allard's theorem since the $L^\infty $ closeness of the graph to a linear map does not give a bound on the density of the graph in $B_1$. On the other hand, the smallness of the excess of a stationary varifold implies its $L^\infty$-closeness to a linear subspace (see for example \cite{SS}).
 In  Proposition \ref{p2} we give a variant of Theorem \ref{T1} which applies for compact sets (not necessarily graphs) that have density strictly less than 2. 
  
The notion of viscosity solution can be easily extended to allow a right hand side $f^\alpha(Du,u,x)$ in the minimal surface system \eqref{MS}. In this case we need to modify the family of comparison functions $H$ and replace in Definition \ref{d1} the right hand side $0$ by a corresponding function depending on $X$ and $L$.  

Finally we remark that in Theorem \ref{T1} we may allow $u$ to be only continuous $\mathcal H^n$ a.e. provided that the viscosity solution notion is understood to hold for the closure of the graph of $u$. It would be interesting to obtain the solvability of the general Dirichlet problem in a ball in the class of viscosity solutions which are continuous $\mathcal H^n$ a.e.

\section{Proof of Theorem \ref{T1}}

We introduce some notation. We denote points in $\R^{n+m}$ by $X=(x,z)$ and by $B_r^n$, $B^m_r$ and $B_r$ the balls of radius $r$ in $\R^n$, $\R^m$ and $R^{n+m}$ respectively. 

We denote by $\mathcal C_\gamma$ the cone of angle $\gamma \in [0, \pi/2)$ around the $x$-subspace
\be\label{C}
\mathcal C_\gamma:= \{X=(x,z) \quad | \quad |z| \le \tan \gamma \, \,  |x|.\}
\ee

\begin{lem}\label{l1}
Assume $u$ satisfies the hypotheses of Theorem \ref{T1}. Then the function
$$w:=\frac 1 \eps |u-l|  $$ 
cannot be touched by above by a function $\varphi$ such that
\be \label{001}
\mathcal M ^+_{c_0,1} (D^2 \varphi) < 0, \quad \| \varphi\|_{C^{1,1}} \le 1,
\ee
with $c_0$ depending on $n$ and $|A|$, and $\mathcal M^+_{\lambda,\Lambda}$ denotes the maximal Pucci operator.
\end{lem}

We recall the definition of $\mathcal M^+_{\lambda,\Lambda}(N)$ for a symmetric matrix $N$: 
$$ \mathcal M^+_{\lambda,\Lambda}(N) = \Lambda |N^+|-\lambda |N^-|.$$
The lemma above states that $w$ is a subsolution to a linear uniformly elliptic equation $$a^{ij}w_{ij} \ge 0$$ at the points where $|\nabla w|$, $|D^2 w|$ are bounded by $1$. This means that the following version of weak Harnack inequality holds (see \cite{S1} or Section 6 in \cite{S2}):

{\bf Weak Harnack inequality:}

 {\it Assume that $w \le 1$ in $B^n_1$ and it cannot be touched by above by functions $\varphi$ that satisfy property $\eqref{001}$. If $$|\{ w< 1-\eta\}| \ge \mu |B^n_1|,$$ for some $\mu$ small, then $$w \le 1-c(\mu) \eta \quad \mbox{in} \quad B^n_{1/2},$$ for some $c(\mu)>0$ small depending on $\mu$, $c_0$, $n$. }

Here we require that $0<\eta  \le \eta_0$ with $\eta_0(\mu)$ sufficiently small, and $|\cdot|$ represents the $n$-dimensional Lebesgue measure.

\

{\it Proof of Lemma \ref{l1}.}
It suffices to show that $$H(X):= |z-l(x)|- \eps \varphi(x), \quad \quad X=(x,z), \quad z \in \R^m,$$
is a comparison function in the cylinder $|z-l(x)| \le \eps$. We do the computation at a point $X_0=(x_0,z_0)$ and, after a translation and then a rotation of the $z$ coordinate, we may assume that $x_0=0$, $l(x_0)=0$, $X_0$ lies on the positive $z_1$-axis. We use
the convexity of $|\cdot|$ and find that near $X_0$
$$|z-l(x)| \ge z_1-l_1(x) + \frac {1}{2 \eps}|z'-l'(x)|^2  $$
which gives
$$H(X) \ge G(X):=G_1(X) + G_2(X)$$
with
$$G_1(X)=z_1-l_1(x) - \eps \varphi(x), \quad \quad G_2(X)= \frac {1}{2 \eps}|z'-l'(x)|^2.$$ 
The function $G$ touches by below $H$ at $X_0$ and it suffices to show that $\triangle _L G(X_0)>0$ for any $n$-dimensional linear subspace $L$ orthogonal to $\nabla G(X_0)$. 

Notice that $G_2$ is convex and at $X_0$ we have 
$$G_2=0, \quad \nabla G_2=0, \quad \quad  \partial^2 _{\xi \xi} G_2 \ge \frac{1}{4 \eps},$$ for any unit direction $\xi$ near the $z'$ subspace. In particular the inequality holds for any unit direction $\xi \perp \nabla H(X_0)$ and (see \eqref{C}) $\xi \notin \mathcal C_{\pi/2-\delta}$ with $\delta$ small depending only on $|Dl|$.

 On the other hand $G_1$ depends only on the $(x,z_1)$ variables and $|D^2G_1| \le \eps$ by the second hypothesis in \eqref{001}. Hence if $L$ is not included in $\mathcal C_{\pi/2-\delta}$ then $$\triangle_L G \ge \triangle _L G_1 + \triangle _L G_2 \ge - n \eps + \frac{1}{4 \eps} >0.$$
Otherwise, the projection $\pi_x$ of $L$ onto the $x$-subspace is a linear map with bounded inverse by a constant depending on $\delta$ thus
$$\triangle_L G \ge \triangle_L G_1 = -\eps \, \, tr (\pi_x^TD^2 \varphi   \, \, \pi_x) \ge - n \eps  \mathcal M^+_{c_0,1}(D^2 \varphi) >0.   $$

\qed

\

We define the oscillation of a function $w: B^n_1 \to \R^m$ as the smallest radius $\rho$ for which the image of $w$ is included in a ball of radius $\rho$,
$$ osc_{B_1}  w=\inf \left \{ \rho\quad | \quad w(B^n_1) \subset B^m_{\rho}(z) \quad \mbox{for some $z$}\right \}.$$

\begin{lem}[Harnack inequality]\label{l2}
Assume $u$ is a viscosity solution of \eqref{MS} and $$osc_{B_1}(u-l) \le \eps.$$ Then 
$$osc_{B_1}(u-l) \le (1-\theta)\eps,$$
for some $\theta >0$ universal.
\end{lem}

\begin{proof}
 We look at the image of the function $$\tilde u:=(u-l)/\eps$$ that maps $B^n_1$ into, say $B^m_1$. We claim that when we restrict $x$ to $B^n_{1/2}$ then the image can be included in a ball of radius $1-\theta$ in $\R^m$. 

Assume that for some $x_0$ in $B^n_{1/2}$, $\tilde u(x_0)$ is $\eta$ close to a point $ \xi \in \p B^m_1$, i.e.
$$\tilde u (x_0) = t \xi, \quad t \in [1-\eta,1].$$
Now we apply Lemma \ref{l1} to the function
$$  w:=\frac 12 |\tilde u + \xi|= \frac{1}{2 \eps}|u-(l- \eps \xi)|,$$
and obtain that $w$ is a subsolution in the sense of Lemma \ref{l1} and
$$|w| \le 1 , \quad \quad w(x_0) \ge 1-\eta.$$ Then by Weak Harnack Inequality for $w$ we find that
\be\label{01}
|\{w>1- C \eta \}| \ge (1-\mu) |B^n_1|,
\ee
for some $C(\mu)$ depending on $\mu$ (we choose $\mu = 1/4$ for example), provided that $\eta$ is chosen sufficiently small. 

Since $$\{ w>1-C \eta\} \subset \{ |\tilde u - \xi| \le C \eta^{1/2}  \},$$
we obtain that $\tilde u$ maps more than $1-\mu$ of the measure of $B^n_1$ in a ball of radius $C \eta^{1/2}$ centered at $\xi$. 

The same argument shows that $\tilde u(B^n_{1/2})$ cannot intersect also $B_\eta (-\xi)$ and the conclusion easily follows.

\end{proof}

The hypothesis that $\Gamma$ is a graph was used only in the last part in the proof above when we said that $\tilde u$ cannot map most of the measure in $B_1^n$ close to $\xi$ and also close to $-\xi$. The graph hypothesis on $\Gamma$ can be replaced by the the following bound on its $\mathcal H^n$-mass
\be\label{02}
\mathcal H^n (\Gamma \cap B_1) \le (2-\delta) \mathcal H^n (B_1^n).
\ee

We state the version of Lemma \ref{l2} in this setting.

\begin{lem}\label{l3}
Assume $\Gamma$ is a viscosity solution (not necessarily a graph) to the minimal surface system and
\be\label{03}
\Gamma \cap B_1 \quad \subset \quad B^n_1 \times B^m_\eps,
\ee  and that \eqref{02} holds. If $\eps \le \eps_0,$ then
\be\label{04}
\Gamma \cap B_{1/2} \quad \subset \quad B^n_1 \times B^m_{(1-\theta)\eps}(z_0),
\ee
with $\theta$ and $\eps_0$ depending on $n$ and $\delta$.
\end{lem}

If $u$ satisfies the hypotheses of Lemma \ref{l2} then its graph $\Gamma$, after a rotation in $\R^{n+m}$, satisfies \eqref{03} and also the conclusion \eqref{04} of Lemma \ref{l3}. We remark however that the rotation of $\Gamma$ might not have the graph property.  

\subsection{Compactness} Let's assume that $\Gamma$ solves \eqref{MS} in the viscosity sense, and that it has the Harnack inequality property of Lemma \ref{l3}. 

We iterate the conclusion one more time and obtain 
$$\Gamma \cap B_{1/4} \subset \quad B^n_1 \times B^m_{(1-\theta)^2\eps}(z_1).$$
By applying this property inductively $l$ times we find that the projection of $\Gamma \cap B_{2^{-l}}$ onto the $z$ variable belongs to a ball of radius $(1-\theta)^{-l} \eps$, as long as $$(1-\theta)^l2^l \eps \le \eps_0.$$

Next, assume we have a sequence of such compact sets $\Gamma_k$ that solve \eqref{MS} in the viscosity sense and satisfy  \eqref{03} for $\eps_k \to 0$. Then, after a dilation of factor $\eps_k^{-1}$ in the second variable, the rescaled sets
$$\tilde \Gamma_k:=\{(x,z)\quad | \quad (x, \eps_k z) \in \Gamma_k\}  $$ 
must converge in the Hausdorff distance, say in $B_{1/2}$, to the graph of a map $\bar u$
$$\bar \Gamma:=\{(x, \bar u(x)) \quad | \quad \bar u : B_{1/2}^n \to B_1^m, \quad \mbox{$\bar u$ is uniformly Holder continuous.}\}  $$

\begin{lem}\label{l35}
$$\triangle \bar u=0.$$ 
\end{lem}

\begin{proof}
Let $h$ be the harmonic function in $B^n_{1/2}$ which is equal to $\bar u$ on $\p B_{1/2}^n$. If $\bar u \ne h$ then
$$|\bar u-h|^2 + \eta |x|^2$$
achieves its maximum away from $\p B^n_{1/2}$ provided that $\eta$ is chosen sufficiently small. This means that for all large $k$
$$H(X)=|f(X)|^2 + \eta |x|^2, \quad \quad f(X):=\eps_k ^{-1}z-h(x),$$
achieves its maximum on $\Gamma_k \cap B_{1/2}$ away from the boundary $\p B_{1/2}$. 

It suffices to show that $H$ is a comparison function in the cylinder $|z| \le \eps_k=: \eps$. Fix a point $X_0=(x_0,z_0)$ and write
$$H(X) = G_1(X) + G_2(X)  $$
with
$$G_1(X)=| f(X_0)|^2 +2 f(X_0) \cdot(f(X)-f(X_0)) + \eta|x|^2,$$
$$G_2(X)=|f(X)-f(X_0)|^2.$$
Notice that $$G_2(X_0)=0, \quad \nabla G_2(X_0)=0, \quad D^2 G_2(X_0) \ge 0,$$
and for any unit direction $\xi$ outside the cone $\mathcal C _\delta$ (see \eqref{C}) we have 
\be\label{05}
\p^2_{\xi \xi} G_2(X_0) \ge c(\delta) \eps^{-2}, \quad \quad \quad \forall \, \xi \notin \mathcal C_\delta.
\ee
 On the other hand $$|D^2 G_1(X_0)| \le C,$$
and $\triangle_L G_1(X_0) \ge c \eta$ when $L$ coincides with the $x$-subspace. By continuity, $$\triangle_L G_1(X_0) > 0 \quad \mbox{ when} \quad  L \subset \mathcal C_\delta,$$
by choosing $\delta$ depending on $\eta$. If $L$ intersects the complement of $\mathcal C_\delta$ then the inequality above is obvious due to \eqref{05}.

\end{proof}

Since $\bar u$ is harmonic in $B_{1/2}^n$, $|\bar u| \le 1$, we obtain $|D^2\bar u|\le C$ in $B_{1/4}^n$ hence
\be\label{06} 
|\bar u - l_0(x)| \le \frac 1 4 \eta \quad \quad \mbox{in} \quad B_{2\eta}^n,
\ee
for some $\eta$ small, universal and $l_0$ is the linear part of $\bar u$ at $0$. Let us assume for simplicity that $0 \in \Gamma_k$ for all $k$'s. Then, \eqref{06} implies that in $B_\eta$ the graphs $\Gamma^k$ can be included in a slight rotation of the cylinder $$ B^n_\eta \times  B^m_{\eps_k \eta/2}.$$
The flatness of the cylinders, i.e. the ratio of the radii of the $m$-dimensional ``height" ball over the $n$-dimensional ``bottom" ball, improved from $\eps_k$ in $B_1$ to $\eps_k/2$ in $B_\eta$. 

This compactness argument together with Lemma \ref{l2} gives the improvement of flatness property of viscosity solutions.

\begin{prop}[Improvement of flatness for graphs]\label{p1}
Assume $u$ is a viscosity solution such that
$$|u-l| \le \eps \quad \mbox{in} \quad B^n_1,$$
where $l(x)=b+ Ax$. Then, there exists a linear map $\tilde l$ such that
$$|u- \tilde l| \le \frac{\eps}{2}\eta \quad \mbox{in} \quad B^n_\eta,$$
with $\eta$ a small fixed constant depending only on $n$, $|A|$.
\end{prop} 

Similarly, if $\Gamma$ is a viscosity minimal set in $B_1$ (not necessarily a graph) and we assume that \eqref{02} holds up to a small scale $c(\delta)$ i.e.
\be\label{07}
\mathcal H^n (\Gamma \cap B_r(x,0)) \le (2-\delta) \mathcal H^n (B_r^n), \quad \quad \forall x \in B^n_1, \quad r \ge c(\delta),
\ee
then Lemma \ref{l3} can be applied several times and the compactness argument holds. We obtain the following version of Proposition \eqref{p1} for viscosity minimal sets.

\begin{prop}\label{p2}
Let $\Gamma$ be a viscosity minimal set in $B_1$ and assume \eqref{07} holds and
$$\pi_z(\Gamma \cap B_1) \subset B^m_\eps,$$
where $\pi_z$ denotes the projection onto the $z$ variable. Then there exists a $\bar z$ subspace obtained by a rotation of the $z$ coordinates, such that
$$\pi_{\bar z} (\Gamma \cap B_\eta) \subset B^m_{\eta \eps /2}(\bar z_0).$$
Here $\eta$, $\delta$, $\eps_0$ are small constant that depend only on $n$. 
\end{prop} 

Proposition \eqref{p1} implies that flat viscosity solutions are $C^{1,\alpha}$ in $B^n_{1/2}$. In order to obtain the $C^{2,\alpha}$ estimates we need to repeat the arguments above and approximate $u$ by harmonic quadratic polynomials $q(x)$ instead of linear functions $l(x)$. We sketch some of the details below. Since $u$ is already $C^{1/\alpha}$, we may rotate coordinates and reduce the $\eps$-smallness hypothesis in Theorem \ref{T1} to the case $l \equiv 0$, i.e. $|u| \le \eps$.
Then the theorem follows by iterating the quadratic improvement of flatness lemma below after performing the necessary rotations of coordinates.

\begin{lem}
Assume $u$ is a viscosity solution such that
$$|u-q| \le \eps \quad \mbox{in} \quad B^n_1, \quad \quad \eps \le \eps_0,$$
where $q:\R^n \to \R^m$ is a harmonic quadratic polynomial with coefficients bounded by $\eps^\beta$ for some $\beta \in (1/2,1)$. Then
$$|u- \tilde q| \le \frac{\eps}{2}\eta^2 \quad \mbox{in} \quad B^n_\eta,$$
with $\tilde q$ a  harmonic quadratic polynomial, and $\eta(n)$, $\eps_0(n,\beta)$ are sufficiently small.
\end{lem}   

For the proof of the lemma it suffices to establish

 a)  Harnack inequality for $\tilde u =(u-q)/\eps$ and,

 b) the convergence of $\tilde u$ to a harmonic map as $\eps \to 0$.  

\

Part a), as in Lemmas \ref{l1} and \ref{l2}, follows once we show that
$$H(x,z)= |z-q(x)|-\eps \varphi(x),$$
is a comparison function provided that $$|\varphi|_{C^{1,1}} \le 1, \quad \triangle \varphi \le - \eps^{2\beta-1}.$$ As in proof af the Lemma \ref{l1} we bound $H$ at $X_0$ by $G_1 + G_2$ with
$$G_1(X)=z_1-q_1(x) - \eps \varphi(x), \quad \quad G_2(X)= \frac {1}{2 \eps}|z'-q'(x)|^2.$$
and at $X_0$, $$|D^2G_1| \le C \eps ^ \beta, \quad \quad \triangle_x G_1=- \eps \triangle \varphi \ge \eps^{2 \beta}.$$
If $L$ makes an angle less than $C \eps^\beta$ with the $x$ subspace,  
$$L\subset \mathcal C_\sigma, \quad \quad \mbox{with} \quad \sigma:=C \eps^\beta$$ then
$$\triangle_L G_1 \ge \triangle_x G_1 - C \eps^\beta \sigma^2 >0.$$ 
Otherwise we can find a unit direction $\xi \in L$ outside $\mathcal C_\sigma$ such that
$$\p^2_{\xi \xi} G_2 \ge \frac {1}{ 8 \eps} \sigma^2 > C \eps^\beta,$$
and the claim follows. 

\

For part b) we argue as in Lemma \ref{l35}. We need to show that
$$H(X)=|f(X)|^2 + \eta |x|^2, \quad \quad f(X):=\eps ^{-1}(z-q(x))-h(x)$$
is a comparison function. 

As before we bound $H$ at $X_0$ by $G_1+ G_2$ where $G_1$ is a linear function of $f(X)$
$$G_1=|f(X_0)|^2 + f(X_0) \cdot (f(X)-f(X_0)) + \eta |x|^2,$$
and
$$ G_2=|f(X)-f(X_0)|^2$$
We have
$$|D^2G_1| \le C \eps^{\beta -1}, \quad \triangle_x G_1 \ge \eta.$$
If $L \subset \mathcal C_\sigma$ with $\mathcal C_\sigma$ defined above then
$$\triangle _L G_1 \ge \triangle_x G_1 -C \eps^{\beta-1} \sigma^2  >0.$$ 
Since for any unit direction $\xi \notin \mathcal C_\sigma$,
$$\p^2_{\xi \xi} G_2 \ge c \eps^{-2}\sigma^2 \ge c \eps^{2 \beta - 2} \gg |D^2 G_1|,$$
the desired conclusion follows.


\begin{thebibliography}{9999}

\bibitem[A]{A} Allard, W. K. {\it On the first variation of a varifold}. Ann. of Math. (2) 95 (1972), 417--491.

\bibitem [F]{F} Fischer-Colbrie D., {\it Some rigidity theorems for minimal submanifolds of the sphere}, Acta
Math. 145 (1980), no. 1-2, 29–46

\bibitem [JX]{JX} Jost J.,  Xin Y. L., {\it  Bernstein type theorems for higher codimension}, Calc. Var. Partial
Differential Equations 9 (1999), no. 4, 277--296.

\bibitem[LO]{LO}
Lawson H. B., Osserman R., {\it Non-existence, non-uniqueness and irregularity of 
solutions to the minimal surface system}, Acta Math. 139 (1977), no. 1-2, 1--17.

\bibitem[S1]{S1} Savin O., {\it Small perturbation solutions for elliptic equations.}
 Comm. Partial Differential Equations 32 (2007),
557--578.


\bibitem[S2]{S2} Savin O.,   {\it Some remarks on the classification of global solutions with asymptotically flat level sets},  arxiv 1610.03448

\bibitem[SS]{SS} Schoen R., Simon L., {\it A new proof of the regularity theorem for rectifiable currents which
minimize parametric elliptic functionals}. Indiana Univ. Math. J. 31 (1982) 415--434.

\bibitem[W]{W}  Wang M-T., {\it The Dirichlet problem for the minimal surface system in arbitrary dimensions
and codimensions,}
Comm. Pure Appl. Math. 57 (2004) 267--281.

\end{thebibliography}
\end{document}